\documentclass[aip, jmp, preprint]{revtex4-1}

\usepackage{amsmath,amssymb,amsthm,amsbsy}
\usepackage[english]{babel}
\usepackage[ansinew]{inputenc}
\usepackage{graphicx}
\usepackage[]{natbib}

\newtheorem{theo}{Theorem}
\newtheorem{prop}[theo]{Proposition}
\newtheorem{lemma}[theo]{Lemma}

\theoremstyle{remark}
\newtheorem{rem}{Remark}
\newtheorem{claim}[theo]{Claim}

\draft

\begin{document}

\title{The critical 1-arm exponent for the ferromagnetic Ising model on the Bethe lattice} 

\author{Markus Heydenreich}
\email[]{E-mail: m.heydenreich@lmu.de}
\altaffiliation{}
\noaffiliation

\author{Leonid Kolesnikov}
\email[]{E-mail: l\_kolesnikov@gmx.de}
\altaffiliation{}
\noaffiliation

\affiliation{Mathematisches Institut, Ludwig-Maximilians-Universität München, Theresienstr. 39, 80333 München, Germany.}



\date{\today}

\begin{abstract}
We consider the ferromagnetic nearest-neighbor Ising model on regular trees (Bethe lattice), which is well-known to undergo a phase transition in the absence of an external magnetic field. The behavior of the model at critical temperature can be described in terms of various critical exponents; one of them is the critical 1-arm exponent $\rho$, which characterizes the rate of decay of the (root) magnetization. The crucial quantity we analyze in this work is the thermal expectation of the root spin on a finite subtree, where the expected value is taken with respect to a probability measure related to the corresponding finite-volume Hamiltonian with a fixed boundary condition. The spontaneous magnetization, which is the limit of this thermal expectation in the distance between the root and the boundary (i.e.\ in the height of the subtree), is known to vanish at criticality. We are interested in a quantitative analysis of the rate of this convergence in terms of the critical 1-arm exponent $\rho$. Therefore, we rigorously prove that $\langle\sigma_0\rangle^+_n$, the thermal expectation of the root spin at the critical temperature and in the presence of the positive boundary condition, decays as $\langle\sigma_0\rangle^+_n\approx n^{-\frac{1}{2}}$ (in a rather sharp sense), where  $n$ is the height of the tree. This establishes the 1-arm critical exponent for the Ising model on regular trees ($\rho=\frac{1}{2}$).
\end{abstract}

\maketitle

\section{Introduction}\label{I}
The ferromagnetic Ising model is one of the most extensively studied models in statistical mechanics. Usually, $\mathbb{Z}^d$ is considered as the underlying graph. Ernst Ising proved in his 1924 thesis\cite{5} that the one-dimensional model does not undergo a phase transition. Using a technique which is now  known as the Peierls argument, Rudolf Peierls\cite{7} showed that for $d\geq 2$ a phase transition occurs in the absence of an external magnetic field. On the other hand, in the presence of a non-zero homogeneous field the model has uniqueness, as was shown by Lee and Yang\cite{6}. Various results were obtained for different types of inhomogeneous external fields (e.g., in Ref.\ \onlinecite{2} and Ref.\ \onlinecite{3}).

Concerning the Ising model on a regular tree (Bethe lattice), Preston\cite{8} showed that there is a phase transition not only in the zero external field, but even in a homogeneous non-zero field - if the field strength does not exceed a certain critical value, which depends both on the structure of the tree and the temperature of the system. However, from now on we deal with the nearest-neighbor Ising model in the absence of an external field.

On $\mathbb{Z}^d$ we consider the expected spin value $\langle\sigma_0\rangle^+_r$ at the center of a ball of radius $r$ with fixed plus spins assigned to its boundary and let $r$ grow towards infinity. As a function of temperature, $\lim_{r\to\infty} \langle\sigma_0\rangle^+_r$ is called spontaneous magnetization and was proven to vanish at criticality. The case $d=2$ was shown in the 1950's by Yang\cite{10} (after Lars Onsager famously announced the formula for the spontaneous magnetization in two dimensions  without giving a derivation), the case $d>2$ only in 2016 by Aizenman, Duminil-Copin and Sidoravicius\cite{1} (though for $d>3$ the proof  was given already in the 80's; see Ref.\ \onlinecite{b}). Handa, Heydenreich and Sakai\cite{4} provided some quantitative analysis of the rate of this convergence for high dimensions ($d\geq 5$). They conjectured that the critical 1-arm exponent for this model equals 1, i.e.\ that at criticality $\langle\sigma_0\rangle^+_r\approx r^{-1}$ as $r\to\infty$. Using  the so-called random-current representation, they were able to show that this critical exponent is bounded from above by 1.

On a regular tree we can proceed analogously: Instead of  a ball of radius $r$, we consider a subtree of height $n$ with a plus boundary condition and look at the limit of the expected spin values $\langle\sigma_0\rangle^+_n$ at the root of the tree as $n$ converges to infinity. By Preston\cite{8} we know that this limit - the spontaneous magnetization - vanishes at criticality, i.e $\lim_{n\to\infty} \langle\sigma_0\rangle^+_n=0$. Again, we are interested in the rate of this convergence. In Ref.\ \onlinecite{4}, the authors make an educated guess that the critical 1-arm exponent equals $\frac{1}{2}$ on regular trees, i.e.\ that at criticality $\langle\sigma_0\rangle^+_n\approx n^{-\frac{1}{2}}$ as $n\to\infty$.
Our main goal is to confirm this conjecture by a rigorous proof.

The ``nice" structure of the Bethe lattice allows us to determine the critical 1-arm exponent using much simpler tools than the ones used to obtain the analogous (partial) result for $\mathbb{Z}^d$ in Ref.\ \onlinecite{4}. To simplify the proofs even further, we consider a minor modification of the Bethe lattice as the underlying graph - by requiring the same number of children for each vertex instead of the same number of neighbors. We expect the qualitative behavior of the model to be invariant under this modification.

Now, we proceed to formally introduce the model and provide the necessary notation.

\subsection{Regular Cayley trees and rooted Cayley trees}
A Bethe lattice or Cayley tree $\Gamma^d$ is an infinite $(d+1)$-regular tree, i.e.\ an infinite connected cycle-free graph, such that every vertex has exactly $d+1$ neighbors. Let $V$ denote the set of vertices  and $L\subset\binom{V}{2}$  the set of edges of $\Gamma^d$. For $x,y\in V$ the distance $d(x,y)$ is given by the length of the shortest path connecting $x$ and $y$. In particular, $d(x,y)=1$ if and only if $x$ and $y$ are neighbors.

We fix some $0\in V$ and call it the root of $\Gamma^d$, then for all $n\in\mathbb{N}_{0}$ we define the $n$-th generation by $
W_n=\{x\in V\vert~ d(0,x)=n\}
$
and the finite volume of height $n$ by 
$
V_n=\{x\in V\vert~ d(0,x)\leq n\}.$ Denote the edges of $\Gamma^d$ restricted to $V_n$ by $L_n$, i.e.\ $L_n=L\cap\binom{V_n}{2}.$
For $x \in W_n$, define the set of children of $x$ by
$
S(x)=\{y\in W_{n+1}\vert~d(x,y)=1\}.
$

Notice that $\vert S(x)\vert=d$ for all $x\in V_{n} \ensuremath{\backslash} \{0\}$, while $\vert S(0) \vert=d+1$. We are interested in a modification of the Bethe lattice, such that $\vert S(x) \vert=d$ for all $x\in V$. To obtain it, it suffices to slightly relax the regularity assumption:

A rooted Cayley tree $\Gamma_r^d=(V,L)$ is an infinite tree, such that all but one vertices have $d+1$ neighbors, while vertex $0\in V$ has $d$ neighbors. We call this distinguished vertex the root of the tree; then we can define the sets $W_n$, $V_n$ and $S(x)$ just as above for the regular Cayley tree. Obviously, every vertex of $\Gamma_r^d$ has exactly $d$ children.

We refer to $ \Gamma^d$ as the regular Cayley tree when we want to emphasize its distinction from the rooted Cayley tree $\Gamma_r^d$.

Next, we introduce the ferromagnetic nearest-neighbor Ising model on the (rooted) Cayley tree.

\subsection{Ferromagnetic nearest-neighbor Ising model}
In our model, spins of value $+1$ or $-1$ are assigned to every vertex of the tree. The set of all possible spin  configurations on $V$ is defined by $\Omega=\{-1,+1\}^V.$ For any subset $U\subset V$, let $\Omega_U=\{-1,+1\}^U$ denote the set of all possible configurations on $U$. For a configuration $\sigma\in\Omega_U$ and a vertex $x\in$ $U$, let $\sigma_x\in\{-1,+1\}$ denote the spin assigned to $x$ in this configuration. For disjoint sets $U$ and $V$, let $\sigma\in \Omega_U$ and $\omega\in \Omega_V$, then the concatenation of $\sigma$ and $\omega$ is defined by
\begin{align*}
(\omega \vee \sigma)_x=
\begin{cases}
     \sigma_x & \text{for } x \in U \\
     \omega_x & \text{for } x \in V.
   \end{cases}
\end{align*}

We define the free Hamiltonian of the ferromagnetic nearest-neighbor Ising model on the volume $V_n$  by
\begin{align*}
H^0_n(\sigma)=-\sum_{\{x,y\}\in L_n}\sigma_x\sigma_y,\quad\sigma\in \Omega_{V_n}.
\end{align*}

Furthermore, using the free Hamiltonian, we define for $\eta\in\Omega$ the Hamiltonian of the ferromagnetic nearest-neighbor Ising model with boundary condition $\eta$ on the volume $V_n$ by
\begin{align*}
H^\eta_n(\sigma)=H^0_n(\sigma)-\sum_{x\in W_n}\sum_{y\in S(x)}\sigma_x\eta_y,\quad\sigma\in\Omega_{V_n}.
\end{align*}

For a positive parameter $\beta$, which is called the inverse temperature of the system, we define a probability measure on $\Omega_{V_n}$ (Gibbs measure) by
\begin{align*}
\mu_n^\eta(\sigma)=\frac{\psi_n^\eta(\sigma)}{Z_n^\eta}:=\frac{e^{-\beta H^\eta_n(\sigma)}}{Z_n^\eta},
\end{align*}
where $Z_n^\eta$ is a normalizing constant given by $Z^\eta_n=\sum_{\sigma\in\Omega}\psi_n^\eta(\sigma)=\sum_{\sigma\in\Omega}e^{-\beta H^\eta_n(\sigma)}$.

When $\eta\equiv +1$, we speak of a plus boundary condition and write $H^\eta_n$ as $H^+_n$, $\psi_n^\eta$ as $\psi_n^+$, $Z_n^\eta$ as $Z_n^+$ and $\mu_n^\eta$ as $\mu_n^+$.

For a function $f$ on $\Omega_{V_n}$, the expected value of $f$ with respect to $\mu_n^+$ is called the thermal expectation of $f$ and is denoted by $\langle f\rangle^+_n$. The crucial quantity we analyze in this thesis is the thermal expectation of the spin value at the root $0\in V$:
\begin{align}\label{spont}
\langle\sigma_0\rangle^+_n=\langle\sigma_0\rangle^+_n(\beta,d)=\sum_{\sigma\in\Omega_{V_n}}\sigma_0\mu^+_n(\sigma).
\end{align}
The thermodynamic limit $\lim_{n\to\infty}\langle\sigma_0\rangle^+_n$ as a function of $\beta$ is known as spontaneous magnetization.

\subsection{The model at criticality}

The critical inverse temperature $\beta_c$ is defined by 
\begin{align}\label{defbc}
\beta_c=\beta_c(d)=\sup\{\beta\geq 0:\lim_{n\to\infty}\langle\sigma_0\rangle_n^+=0\}
\end{align}
and is, as already mentioned, non-trivial by Preston\cite{8}. Furthermore, Rozikov\cite{9} showed that
\begin{align}\label{betac}
\beta_c(d)=\tanh^{-1}\left(\frac{1}{d}\right).
\end{align}

The behavior of the Ising model at the critical temperature can best be described by critical exponents. A few examples are the critical exponents $\beta$, $\gamma$ and $\delta$ (see Ref.\ \onlinecite{a} and Ref.\ \onlinecite{b} for definitions), whose values on a regular tree are the same as the mean-field values $\beta=\frac{1}{2}$, $\gamma=1$ and $\delta=3$; cf.\ Chapter 4.8 in Baxter\cite{e}. 
For many statistical physics models, there is an upper critical dimension $d_c$ such that the corresponding model on $\mathbb{Z}^d$ with $d>d_c$ has the same values for the critical exponents as on the tree. Indeed, Aizenman and Fernández\cite{a,b} proved that on $\mathbb{Z}^d$ the so-called ``bubble condition" (which is a summability condition on the 2-spin expectations at criticality) is sufficient to establish the same values for these exponents. Using the infrared bound established by Fröhlich, Simon and Spencer\cite{c}, one can show that the bubble condition is satisfied for the nearest-neighbor Ising model on the lattice in dimension $d>4$ (see Ref.\ \onlinecite{a}), so that for these models the critical exponents indeed assume the mean-field values. 
However, for the critical exponent we deal with in this work, that seems not to be the case.

An important concept of statistical physics is that such critical exponents are universal in the sense that they are invariant under various modifications of the underlying graph (even though some of those might change $\beta_c$). 
For example, Dommers, Giardinà and van der Hofstad\cite{d} proved that the critical exponents $\beta$, $\gamma$ and $\delta$ on \emph{random} trees (as well as a class of tree-like random graphs) have the same value as on the \emph{deterministic} Cayley tree. 
It would be interesting to see whether similar generalizations are valid for the 1-arm exponent.

\subsection{The main result}
In this work we consider the 1-arm exponent $\rho$, which characterizes the decay of $\langle\sigma_0\rangle_n^+$. To be more precise, we call $\rho$ the critical 1-arm exponent, if there exist constants $c, C>0$, such that 
\begin{align}\label{rho}
cn^{-\rho}<\langle\sigma_0\rangle_n^+<Cn^{-\rho}
\end{align}
for $\beta=\beta_c$ and for all  $n\in \mathbb{N}$.

Notice that this is a fairly strong notion of a critical exponent; weaker modes of convergence have been considered in the literature.

Our main result is the following:
\begin{theo}[Critical 1-arm exponent]\label{thm1}~\\
For $d\geq 2$ and $\beta=\beta_c$ there exist constants $c,C>0$ such that for all $n\in\mathbb{N}$\[ cn^{-\frac{1}{2}}<\langle\sigma_0\rangle_n^+<Cn^{-\frac{1}{2}},\] i.e.\ the critical 1-arm exponent of the ferromagnetic nearest-neighbor Ising model on the (rooted) Cayley tree $\rho$ (as defined in \eqref{rho}) equals $\frac{1}{2}$.
\end{theo}

Recall that Handa et al. conjectured that the critical 1-arm exponent on $\mathbb{Z}^d$ is equal to 1 in high dimensions. The factor $\frac{1}{2}$ with respect to the exponent on the Cayley tree can be explained as arising from a random embedding of the tree into $\mathbb{Z}^d$ (mind that the expected end-to-end distance of a path with length $n$ randomly embedded into $\mathbb{Z}^d$ is $\sqrt{n}$). Other critical exponents such as $\beta$, $\gamma$ and $\delta$ are defined via cluster sizes, and hence no ``correction" between the lattice and the tree occurs.

To prove Theorem \ref{thm1} we first establish a new recursive relation for the sequence $(\langle\sigma_0\rangle^+_n)_{n\in\mathbb{N}}$ in Section \ref{II} and then analyze this relation in Section \ref{III} - to finally conclude the proof in Section \ref{IV}.

\section{Recursive representation of the thermal spin expectations}\label{II}
First, we would like to express  the expected spin values $\langle\sigma_0\rangle^+_n$ in terms of a recursively defined sequence $(\frac{y_n}{{x_n}})_{n\in\mathbb{N}_0}$. This recursive representation is not only useful for numerical computation, but also more suited for some elementary methods of mathematical analysis than the initial definition \eqref{spont}. The proofs presented in the following section heavily rely upon this fact. To define the sequence $(\frac{y_n}{{x_n}})_{n\in\mathbb{N}_0}$, we introduce the following notation: ~

~\\Let 
\begin{align}\label{xy0}
\begin{aligned}
x_0&=x_0(\beta,d):=(e^{\beta(d+1)}+e^{-\beta(d+1)})^d,\\ y_0&=y_0(\beta,d):=(e^{\beta(d-1)}+e^{-\beta(d-1)})^d.
\end{aligned}
\end{align}
Then,  for all $n\in \mathbb{N}$ define 
\begin{align}\label{xyn}
\begin{aligned}
x_n = x_n(\beta,d):=(e^{\beta}x_{n-1}+e^{-\beta}y_{n-1})^d,\\y_n = y_n(\beta,d):=(e^{\beta}y_{n-1}+e^{-\beta}x_{n-1})^d. 
\end{aligned}
\end{align}

\begin{theo}[Recursive representation]\label{thm2}~\\With notation as in \eqref{xy0} and \eqref{xyn}:\begin{align*}
  \langle\sigma_0\rangle^+_{n+1}=\frac{x_n-y_n}{x_n+y_n}=\frac{2}{\frac{y_n}{x_n}+1}-1
 \end{align*}
for all $\beta>0,~d\geq 2\text{ and }n\in \mathbb{N}_0$.
\end{theo}
We formulate this result as a theorem, because it is of great generality and of own independent value. The recursive representation can be used to analyze the spontaneous magnetization not only at criticality, but at any positive temperature. For example, using argumentation similar to the one presented in Section \ref{III}, it can provide an alternative proof of \eqref{betac}.

The recursive relation is tailored for our analysis in Section \ref{III}. However, equivalent versions of such a relation have been used in the literature. 
For example, it was used in Baxter's book\cite{e} in order to derive a number of other critical exponents, and also by Bissacot, Endo and van Enter\cite{f} in order to investigate boundary fields corresponding to compatible measures. 

%

\begin{proof}[Proof of Theorem \ref{thm2}]
We rewrite the definition of $\langle\sigma_0\rangle^+_{n+1}$ (see \eqref{spont}) by representing configurations on the finite volume $V_{n+1}$ as concatenations of configurations on $V_n$ and configurations on $W_{n+1}$. We then proceed to sort the exponential terms accordingly:

\begin{align*}
\langle\sigma_0\rangle^+_{n+1}&=\frac{1}{Z^+_{n+1}}\sum_{\sigma\in \Omega^{V_{n+1}}}\sigma_0   \psi^+_{n+1}(\sigma)
=\frac{1}{Z^+_{n+1}}\sum_{\sigma\in \Omega^{V_n}}\sum_{\omega\in \Omega^{W_{n+1}}}\sigma_0   \psi^+_{n+1}(\sigma\vee\omega)
\\
&= \frac{1}{Z^+_{n+1}} \sum_{\sigma\in \Omega^{V_n}}\sigma_0e^{\beta\sum_{\{x,y\}\in L_n}\sigma_x\sigma_y}\sum_{\omega\in \Omega^{W_{n+1}}}e^{\beta\sum_{x\in W_n}\sum_{y\in S(x)}\sigma_x\omega_y}e^{\beta\sum_{x\in W_{n+1}}d\omega_x}
\\
&=\frac{1}{Z^+_{n+1}} \sum_{\sigma\in \Omega^{V_n}}\sigma_0e^{\beta\sum_{\{x,y\}\in L_n}\sigma_x\sigma_y}\sum_{\omega\in \Omega^{W_{n+1}}}\prod_{x\in W_n}\prod_{y\in S(x)}e^{\beta(\sigma_x\omega_y+d\omega_y)}
\\
&\stackrel{(\star)}{=}\frac{1}{Z^+_{n+1}} \sum_{\sigma\in \Omega^{V_n}}\sigma_0e^{\beta\sum_{\{x,y\}\in L_n}\sigma_x\sigma_y}\prod_{x\in W_n}\left(\sum_{u\in\{-1,+1\}}e^{\beta(\sigma_xu+du)}\right)^d.
\end{align*}

Since the configurations on $W_{n+1}$ are just all possible combinations of spin values $+1$ and $-1$ on $W_{n+1}=\bigcup_{x\in W_n}S(x)$, where the union is disjoint and every set of successors S(x) consists of $d$ elements, $(\star)$ is just a combinatorical consideration (the same argument is used to derive the recursion for ``boundary fields" by Bissacot, Endo and van Enter\cite{f}).

Now we define for $a\in\{-1,+1\}$:
\begin{align*}
G(a)=G(a,\beta,d):=\left(\sum_{u\in\{-1,+1\}}e^{\beta(au+du)}\right)^d,
\end{align*}
thus $G(+1)=x_0$ and $G(-1)=y_0$ (see \eqref{xy0}).

Then,
\begin{align*}
\langle\sigma_0\rangle^+_{n+1}&=\frac{1}{Z^+_{n+1}} \sum_{\sigma\in \Omega^{V_n}}\sigma_0e^{\beta\sum_{\{x,y\}\in L_n}\sigma_x\sigma_y}\prod_{x\in W_n}G(\sigma_x)\\
&=\frac{1}{Z^+_{n+1}} \sum_{\sigma\in \Omega^{V_{n-1}}}\sigma_0e^{\beta\sum_{\{x,y\}\in L_{n-1}}\sigma_x\sigma_y}\sum_{\omega\in W_n}\left(e^{\beta\sum_{x\in W_{n-1}}\sum_{y\in S(x)}\sigma_x\omega_y}\prod_{x\in W_{n-1}}\prod_{y\in S(x)}G(\omega_y)\right)\\
&=\frac{1}{Z^+_{n+1}} \sum_{\sigma\in \Omega^{V_{n-1}}}\sigma_0e^{\beta\sum_{\{x,y\}\in L_{n-1}}\sigma_x\sigma_y}\prod_{x\in W_{n-1}}\left(\sum_{u\in \{-1,+1\}}e^{\beta\sigma_x u}G(u)\right)^d,
\end{align*}
where the last equality holds by the same combinatorical argument as $(\star)$.

We continue this procedure inductively. We define for $a\in\{-1,+1\}$:
\begin{align*}
G_0(a)&:=G(a)\\
\text{and for } 0<k\leq n:~
G_k(a)&:=\left(\sum_{u\in \{-1,+1\}}e^{\beta a u}G_{k-1}(u)\right)^d,
\end{align*}
thus $G_k(+1)= x_k$ and $G_k(-1)=y_k$ (see \eqref{xyn}).

With that,\newline \[\langle\sigma_0\rangle^+_{n+1}=\frac{1}{Z^+_{n+1}} \sum_{\sigma\in \Omega^{V_{n-1}}}\sigma_0e^{\beta\sum_{\{x,y\}\in L_{n-1}}\sigma_x\sigma_y}\prod_{x\in W_{n-1}}G_1(\sigma_x)\]\newline and then, by induction, for $ 0<k\leq n$ \newline \[\langle\sigma_0\rangle^+_{n+1}=\frac{1}{Z^+_{n+1}} \sum_{\sigma\in \Omega^{V_{n-k}}}\sigma_0e^{\beta\sum_{\{x,y\}\in L_{n-k}}\sigma_x\sigma_y}\prod_{x\in W_{n-k}}G_k(\sigma_x).\]

In particular,
\begin{align}\label{rooted}
\langle\sigma_0\rangle^+_{n+1}&=\frac{1}{Z^+_{n+1}} \sum_{\sigma\in \Omega^{V_1}}\sigma_0e^{\beta\sum_{\{x,y\}\in L_1}\sigma_x\sigma_y}\prod_{x\in W_1}G_{n-1}(\sigma_x)\nonumber \\
&=\frac{1}{Z^+_{n+1}} \sum_{\sigma_0 \in \{-1,+1\}}\sigma_0 \sum_{\omega \in \Omega^{W_1}}e^{\beta\sum_{x\in W_1}\sigma_0\omega_x}\prod_{x\in W_1}G_{n-1}(\sigma_x) \nonumber\\
&=\frac{1}{Z^+_{n+1}} \sum_{\sigma_0 \in \{-1,+1\}}\sigma_0 \sum_{\omega \in \Omega^{W_1}}\prod_{y\in S(0)}e^{\beta\sigma_0\omega_y}\prod_{y\in S(0)}G_{n-1}(\sigma_y)\nonumber \\
&=\frac{1}{Z^+_{n+1}} \sum_{\sigma_0 \in \{-1,+1\}}\sigma_0\left(\sum_{u\in \{-1,+1\}}e^{\beta\sigma_0 u}G_{n-1}(u)\right)^d \\
\nonumber &=\frac{G_n(+1)-G_n(-1)}{Z^+_{n+1}}=\frac{G_n(+1)-G_n(-1)}{G_n(+1)+G_n(-1)}=\frac{x_n-y_n}{x_n+y_n}.
\end{align}
The second equality in the claim of the theorem follows immediately by a simple transformation of the last fraction.
\end{proof}~

\begin{rem}
Such a representation of $\langle\sigma_0\rangle^+_n$ seems quite natural: The reader will notice that $\frac{x_n}{x_n+y_n}$ can be interpreted as the probability of the root taking the spin value $+1$, $\frac{y_n}{x_n+y_n}$ as the probability of $-1$.
\end{rem}

\begin{rem}
Naturally, we can use the same procedure to find a recursive representation for regular Cayley trees.  The proof would be exactly the same - up to the last step of the induction, where instead of \eqref{rooted} we would get
\begin{align*}
\langle\sigma_0\rangle^+_{n+1}=\frac{1}{Z^+_{n+1}} \sum_{\sigma_0 \in \{-1,+1\}}\sigma_0\left(\sum_{u\in \{-1,+1\}}e^{\beta\sigma_0 u}G_{n-1}(u)\right)^{d+1}=\frac{x_n\sqrt[d]{x_n}-y_n\sqrt[d]{y_n}}{x_n\sqrt[d]{x_n}+y_n\sqrt[d]{y_n}},
\end{align*}
since in that case $S(0)$ consists of $d+1$ elements. We see  that the recursive representation is ``nicer" for the rooted Cayley tree, which is why we consider this modification of the underlying graph in the first place.
\end{rem}

\begin{rem}As an immediate consequence of Theorem \ref{thm2}: \begin{align}\label{thd}
\lim_{n\to\infty}\langle\sigma_0\rangle^+_n=0\text{ if and only if }\lim_{n\to\infty}\frac{y_n}{x_n}=1.
\end{align}
\end{rem}

\begin{rem}\label{rem3}
Also, notice the following fact, which will be used implicitely throughout the paper: For all $\beta>0,~d\geq 2$ and $n\in \mathbb{N}_0$
\begin{align*}
\frac{y_n}{x_n}<1.
\end{align*}
This inequality holds by Theorem \ref{thm2}, since $\langle\sigma_0\rangle^+_n$ is always positive. Intuitively, the latter should be clear: The influence of the positive boundary certainly remains on any finite volume of the lattice, no matter how weak the interactions between the neighbors are or how big the distance between the root and the boundary generation gets (though it might vanish after taking the thermodynamic limit). Alternatively, the claim of Remark \ref{rem3} can be proven via simple induction using just the definition of the recursive sequence $(\frac{y_n}{{x_n}})_{n\in\mathbb{N}_0}$.
\end{rem}

\section{Analyzing the recursion at criticality}\label{III}
Now, we consider Cayley trees with arbitrary structure (i.e.\ $d\geq 2$), but focus on the critical case $\beta=\beta_c$, where $\beta_c$ is defined as in \eqref{defbc}. Recall, that on the $(d+1)$-regular Cayley tree $\Gamma^d$ the critical inverse temperature is known to be $\tanh^{-1}({\frac{1}{d}})$ (see \eqref{betac} and Ref.\ \onlinecite{9}). Notice that considering the rooted Cayley tree as the underlying graph does not change this critical value.

 With the recursive representation from Theorem \ref{thm2} at hand, we can determine the critical 1-arm exponent for our model by analyzing the asymptotic behavior of $(\frac{y_n}{{x_n}})_{n\in\mathbb{N}_0}$ at criticality (i.e.\ for $\beta=\tanh^{-1}(\frac{1}{d})$). That analysis is provided by the following result, which is the main ingredient to the proof of Theorem \ref{thm1}:
\begin{lemma}[]\label{lemma}~\\
$\text{For any }d\geq 2,~\beta=\beta_c=\tanh^{-1}(\frac{1}{d})\text{ and for all }n\in\mathbb{N}$\[\quad\quad 1-\left(1-\frac{y_0}{x_0}\right)\frac{1}{\sqrt{n}}\geq\frac{y_{n-1}}{x_{n-1}}\geq 1-\left(\frac{\sqrt{6}d}{\sqrt{d^2-1}}\right)\frac{1}{\sqrt{n}}.\]
\end{lemma}

\subsection{Auxiliary results}\label{IIIa}
To prove Lemma \ref{lemma}, we want to establish four auxiliary results, which we call propositions. The first one is not only useful for our proof, but also provides a method for efficient numerical computation of $(\frac{y_n}{{x_n}})_{n\in\mathbb{N}_0}$.

Notice that the sequences $(x_n)_{n\in\mathbb{N}}$ and $(y_n)_{n\in\mathbb{N}}$ grow very fast in $n$, especially for $d$ large, since $d$ appears in every step of the recursion \eqref{xyn} as an exponent. The straightforward approach - computing  $x_n$ and $y_n$ first, then determining their ratio - is thus highly inefficient and practically impossible for high values of $n$. Luckily, we can obtain the first $n$ elements of $(\frac{y_n}{{x_n}})_{n\in\mathbb{N}}$ directly from the starting point $\frac{y_0}{{x_0}}$ - without computing $x_n$ and $y_n$ - using an iterated function, which we introduce in the following:

~\\Let
\begin{align}\label{g}
g(x):=\left(x+\frac{1-x^2}{b+x}\right)^d,
\end{align}
where $b:=e^{2\beta}$. Furthermore, for a function $f$ let $f^n$ be the function obtained by composing $f$ with itself $n$ times.

\begin{prop}[Iterated function representation]~\\
With notation as in \eqref{g}, for all $\beta>0,~d\geq 2$ and $n\in \mathbb{N}_0:$
\begin{align*}
\frac{y_n}{{x_n}}=g^n\left(\frac{y_0}{x_0}\right).
\end{align*}
\end{prop}
\begin{proof} Using some basic transformations, we get:
\begin{align*}
\frac{y_n}{x_n}&=\left(\frac{e^{\beta}y_{n-1}+e^{-\beta}x_{n-1}}{e^{\beta}x_{n-1}+e^{-\beta}y_{n-1}}\right)^d=\left(\frac{y_{n-1}}{x_{n-1}}+\frac{x_{n-1}^2-y_{n-1}^2}{x_{n-1}(bx_{n-1}+y_{n-1})}\right)^d\\&=
\left(\frac{y_{n-1}}{x_{n-1}}+\frac{1-\left(\frac{y_{n-1}}{x_{n-1}}\right)^2}{b+\frac{y_{n-1}}{x_{n-1}}}\right)^d=g\left(\frac{y_{n-1}}{x_{n-1}}\right).
\end{align*}
\end{proof}

The following three propositions are rather technical results, which are stated in a slightly more general setting than the lemma they are used to prove.

First, we want to establish a sufficient condition on $\frac{y_{n-1}}{x_{n-1}}$ for $\frac{y_n}{{x_n}}$ to be bounded by a term of the form $1-\frac{k}{\sqrt{m}}$ (to use it in an inductive argument in the proof of Lemma 3). Such a condition should be established for all $n,m \in\mathbb{N}$ and some suitable set of positive values $k$. It is provided by the following result.

\begin{prop}\label{prop5}~\\
Let $d\geq 2$, $n\in\mathbb{N}$ and $k>0$, such that $k^2\leq n$. Further, let $0<y<x$. Assume:
\begin{align*}
\frac{y}{x}\stackrel{(\geq)}{\leq} K^d_{\sqrt{n}}(k),
\end{align*}  where 
\begin{align}\label{kcrk}
K^d_{\sqrt{n}}(k):=\frac{-(d+1)\left(1-\frac{k}{\sqrt{n}}\right)^{\frac{1}{d}}+(d-1)}{\left(d-1\right)\left(1-\frac{k}{\sqrt{n}}\right)^{\frac{1}{d}}-(d+1)}.
\end{align}
Then - using notation \eqref{g} - for $\beta=\beta_c=tanh^{-1}\left(\frac{1}{d}\right)$, i.e.\ $b=e^{2\beta_c}=\frac{d+1}{d-1}$, the following inequality holds:
\begin{align}\label{elv}
g\left(\frac{y}{x}\right)\stackrel{(\geq)}{\leq} 1-\frac{k}{\sqrt{n}}.
\end{align}
\end{prop}
\begin{proof}
First, notice that 
\begin{align*}
g(z):=\left(z+\frac{1-z^2}{b+z}\right)^d=\frac{\left( bz+1\right)^d}{\left( b+z\right)^d}.
\end{align*}

Using this simple identity, for $z=\frac{y}{x}\in\left(0,1\right)$ and an arbitrary (inverse) temperature $\beta>0$, \eqref{elv} is equivalent to
\begin{align*}
&\frac{\left( bz+1\right)^d}{\left( b+z\right)^d}\stackrel{(\geq)}{\leq}1-\frac{k}{\sqrt{n}}.
\end{align*}

Again, this holds if and only if
\begin{align*}
\left( bz+1\right)^d-\left( b+z\right)^d+\frac{k}{\sqrt{n}}\left( b+z\right)^d\stackrel{(\geq)}{\leq}0.
\end{align*}

At criticality, substituting $b=b_c=\frac{d+1}{d-1}$ (notice: $\tanh^{-1}(\frac{1}{d})=\log\left(\sqrt{\frac{d+1}{d-1}}\right)$  for $d>1$) yields the equivalence of \eqref{elv} and
\begin{align}\label{nst}
\nonumber&f(z):=\left[(d+1)z+(d-1)\right]^d-\left[(d-1)z+(d+1)\right]^d\\&+\frac{k}{\sqrt{n}}\left[(d-1)z+(d+1)\right]^d\stackrel{(\geq)}{\leq}0.
\end{align}

The claim of the proposition follows by the fact that $K^d_{\sqrt{n}}(k)$ is the only potentially positive zero (in the sense that any other zero is non-positive) of $f(z)$, as defined in $\eqref{nst}$:

\item[$\bullet$]\quad If $K^d_{\sqrt{n}}(k)\geq z>0$, then $f(z)\leq 0$, since $f(0)=(d-1)^d-(d+1)^d+\frac{k}{\sqrt{n}}(d+1)^d<0$ (notice the additional implicite assumption on $n$ and $k$ implied by the condition $K^d_{\sqrt{n}}(k)>0$) and since there is no positive zero of $f(z)$ smaller than $K^d_{\sqrt{n}}(k)$.

\item[$\bullet$]\quad If $K^d_{\sqrt{n}}(k)\leq z<1$, then $f(z)\geq 0$, since $f(1)=\frac{k}{\sqrt{n}}2^dd^d>0$ and since there is no positive zero of $f(z)$ larger than $K^d_{\sqrt{n}}(k)$.
\end{proof}

We now want to find upper and lower bounds of the form $1-\frac{k}{\sqrt{n-1}}$ for the auxiliary sequence $K^d_{\sqrt{n}}(k)$ and suitable values of $k=k(d)$. We start with an upper bound.
\begin{prop}\label{prop6}~\\
Let $d\geq 2$ and $n\in\mathbb{N}$, such that $n\geq \frac{6d^2}{d^2-1}$. Then, with notation as in $\eqref{kcrk}$:
\begin{align}\label{11}
K^d_{\sqrt{n}}\left(\frac{\sqrt{6}d}{\sqrt{d^2-1}}\right)\leq 1-\left(\frac{\sqrt{6}d}{\sqrt{d^2-1}}\right)\frac{1}{\sqrt{n-1}}.
\end{align}
\end{prop}

\begin{proof}
Using the definition $\eqref{kcrk}$ of $K^d_{\sqrt{n}}(k)$ and the fact that $\left(d-1\right)\left(1-\frac{k}{\sqrt{n}}\right)^{\frac{1}{d}}-(d+1)<0$, we can establish the following equivalence for all $k>0$:
\begin{align*}
&K^d_{\sqrt{n}}(k)\leq 1-\frac{k}{\sqrt{n-1}}
\end{align*}
if and only if
\begin{align*}
&-(d+1)\left(1-\frac{k}{\sqrt{n}}\right)^\frac{1}{d}+(d-1)-(d-1)\left(1-\frac{k}{\sqrt{n}}\right)^\frac{1}{d}+(d+1)+\\&\frac{k}{\sqrt{n-1}}\left((d-1)\left(1-\frac{k}{\sqrt{n}}\right)^\frac{1}{d}-(d+1)\right)\geq 0.
\end{align*}

Substitution $k=\frac{\sqrt{6}d}{\sqrt{d^2-1}}$ now yields the equivalence of inequality \eqref{11} and
\begin{align*}
&2d-2d\left(1-\frac{\sqrt{6}d}{\sqrt{(d^2-1)n}}\right)^\frac{1}{d}+\frac{\sqrt{6}d\left((d-1)\left(1-\frac{\sqrt{6}d}{\sqrt{(d^2-1)n}}\right)^\frac{1}{d}-(d+1)\right)}{\sqrt{d^2-1}\sqrt{n-1}}\geq 0.
\end{align*}

Now we devide the last inequality by $d$ and substitute $n=r^2\frac{6d^2}{d^2-1}$, where $r\geq 1$ (since $n\geq \frac{6d^2}{d^2-1}$ by assumption). We get the equivalence of \eqref{11} and
\begin{align*}
&\frac{\sqrt{6}\left((d-1)\left(1-\frac{1}{r}\right)^\frac{1}{d}-(d+1)\right)}{\sqrt{d^2-1}\sqrt{\frac{6d^2r^2}{d^2-1}-1}}-2\left(1-\frac{1}{r}\right)^\frac{1}{d}+2\geq 0.
\end{align*}

By simple transformation that is again equivalent to the following inequality:
\begin{align}\label{22}
F_d(r):=2\sqrt{d^2(6r^2-1)+1}\left(\left(1-\frac{1}{r}\right)^\frac{1}{d} - 1\right) + \sqrt{6}\left((-d + 1)\left(1-\frac{1}{r}\right)^\frac{1}{d} + (d + 1)\right)\leq 0.
\end{align}

We will show that the inequality \eqref{22} holds on $[1, \infty)$ for all $d\geq 2$:

First, notice that $F_d(1)=\sqrt{6}(d+1)-2\sqrt{5d^2+1}\leq 0$ for $d>1$.

For $r>1$, we can show the inequality by using another substitution: $z=(1-\frac{1}{r})^\frac{1}{d}$, i.e  $r=(1- z^d)^{-1}$, where $0<z<1$. Via this substitution \eqref{22} turns into 
\begin{align*}
&2(z-1)\sqrt{\frac{6d^2}{(z^d-1)^2}-d^2+1}+\sqrt{6}[(-d+1)z+(d+1)]\leq 0,
\end{align*}
which is equivalent to
\begin{align*}
2(1-z)\sqrt{\frac{6d^2}{(z^d-1)^2}-d^2+1}\geq \sqrt{6}[(-d+1)z+(d+1)].
\end{align*}
Taking the square of both sides (which are positive under our assumptions) now yields the equivalent inequality
\begin{align*}
4(1-z)^2\left(\frac{6d^2}{(z^d-1)^2}-d^2+1\right)-6[(-d+1)z+(d+1)]^2\geq 0.
\end{align*}

Multiplying this last inequality with $(z^d-1)^2$, we obtain the equivalence of inequality \eqref{22} and
\begin{align}\label{Gd}
\nonumber&G_d(z):= 4(1-z)^2\left(-d^2 (z^{2d} - 2 z^d -5) + (z^d - 1)^2\right)\\&-6(z^d-1)^2[(-d+1)z+(d+1)]^2\geq 0
\end{align}
for all $0<z<1$ and $d\geq 2$.

To see that \eqref{Gd} holds, notice that for every $d\geq 2$ the polynomial $G_d(z)$, which is of order $2d+2$, can be written as follows:
\begin{align*}
G_d(z)=(1-z)^5\sum_{i=0}^{2d-3}\alpha_i z^i,
\end{align*}
where $\alpha_i\in\mathbb{N}$ for $0\leq i\leq 2d-3$, in particular every coefficient $\alpha_i$ is (strictly) positive (for example, $G_2(z)=(1-z)^5(18z+30)$ and $G_3(z)=(1-z)^5(56z^3+120z^2+168z+88)$).

This observation immediately yields inequality \eqref{Gd} and thus \eqref{22} holds on $[1, \infty)$ for all $d\geq 2$, which is, as stated above, equivalent to the claim of the proposition.
\end{proof}

After dealing with an upper bound in Proposition \ref{prop6}, we proceed to show a lower bound for $K^d_{\sqrt{n}}\left(k\right)$ of the same form (i.e.\ $1-\frac{k}{\sqrt{n-1}}$), but at a different value of $k=k(d)$. We choose this value to be $1-\frac{y_0}{x_0}$ at criticality (i.e.\ $y_0=y_0(\beta_c,d)$ and $x_0=x_0(\beta_c,d)$).

\begin{prop}\label{prop7}~\\
Again with notation as in $\eqref{kcrk}$, for all $d\geq 2$, $n\in\mathbb{N}$ and $\beta=\beta_c=\tanh^{-1}(\frac{1}{d})$:
\begin{align*}
K^d_{\sqrt{n}}\left(1-\frac{y_0}{x_0}\right)\geq 1-\left(1-\frac{y_0}{x_0}\right)\frac{1}{\sqrt{n-1}}.
\end{align*}
\end{prop}

\begin{proof}
We will, in fact, show the inequality
\begin{align}\label{k1}
K^d_{\sqrt{n}}(k)\geq 1-\frac{k}{\sqrt{n-1}}
\end{align}
for all $k\leq 1$.

As in the proof of Proposition \ref{prop6}, the inequality \eqref{k1} is equivalent to
\begin{align*}
&-(d+1)\left(1-\frac{k}{\sqrt{n}}\right)^\frac{1}{d}+(d-1)-(d-1)\left(1-\frac{k}{\sqrt{n}}\right)^\frac{1}{d}+(d+1)\\
&+\frac{k}{\sqrt{n-1}}\left((d-1)\left(1-\frac{k}{\sqrt{n}}\right)^\frac{1}{d}-(d+1)\right)\leq 0,
\end{align*}
therefore also to
\begin{align*}
\frac{d(k-2\sqrt{n-1})}{\sqrt{n-1}}\left(\left(\frac{\sqrt{n}-k}{\sqrt{n}}\right)^\frac{1}{d}-1\right)-\frac{k}{\sqrt{n-1}}\left(\left(\frac{\sqrt{n}-k}{\sqrt{n}}\right)^\frac{1}{d}+1\right)\leq 0
\end{align*}
and, finally, to
\begin{align}\label{lst}
Q_{d,k}(n):=\left(d\left(k-2\sqrt{n-1}\right)-k\right) \left(\left(\frac{\sqrt{n}-k}{\sqrt{n}}\right)^{1/d}-1\right)-2k\leq 0.
\end{align}

Notice that to show the last inequality, it suffices to prove the following claim.

\begin{claim}\label{claim} With notation as in \eqref{lst}, the function $n\mapsto Q_{d,k}(n)$ converges towards zero ($Q_{d,k}(n)\to 0$ as $n\to\infty$) and is increasing on $[1,\infty ]$ for any $d\geq 2$ and $0<k<1$.
\end{claim}

\begin{proof}[Proof of Claim \ref{claim}]

\item[$\bullet$]\quad By \eqref{lst}, $Q_{d,k}(n)\to 0$ as $n\to\infty$ can be rewritten as \[
\lim_{n\to\infty}\left(d\left(k-2\sqrt{n-1}\right)-k\right) \left(\left(\frac{\sqrt{n}-k}{\sqrt{n}}\right)^{1/d}-1\right)=2k,
\]which follows directly from \[
\lim_{n\to\infty}\left(\frac{\sqrt{n} - k}{\sqrt{n}}\right)^{\frac{1}{d}}=1\] and \[  \lim_{n\to\infty} \sqrt{n}\left(\left(\frac{\sqrt{n} - k}{\sqrt{n}}\right)^{\frac{1}{d}} - 1\right)= -\frac{k}{d}.
\]

\item[$\bullet$] \quad To show that $n\mapsto Q_{d,k}(n)$ is increasing on $[1, \infty]$, it suffices to prove that $\frac{d}{dn}~Q_{d,k}(n)>0$ on that interval; we differentiate \eqref{lst} to obtain
\begin{align}\label{rk}
\nonumber R_{d,n}(k):=\frac{d}{dn}~Q_{d,k}(n)&=\frac{1}{d}\left(\frac{1}{2n} - \frac{\sqrt{n}-k}{2n^{1.5}}\right)\left(d(k-2\sqrt{n - 1}) - k\right) \left(\frac{\sqrt{n} - k)}{\sqrt{n}}\right)^{\frac{1}{d}-1}\\&-\frac{d}{\sqrt{n-1}}\left(\left(\frac{\sqrt{n} - k)}{\sqrt{n}}\right)^{\frac{1}{d}}-1 \right).
\end{align}

First, notice that 
\begin{align}\label{r0}
R_{d,n}(0)=0,
\end{align}
i.e\ $\frac{d}{dn}~Q_{d,k}(n)=0$ at $k=0$, for any $d\geq 2$ and $n\geq 1$.

Next, we will show that  $k\mapsto R_{d,n}(k)$ is increasing on $[0,1)$ for any $d\geq 2$ and $n\geq 1$ - by showing $\frac{d}{dk}~R_{d,n}(k)>0$ on $[0,1)$. We differentiate \eqref{rk} to obtain 
\begin{align*}
\frac{d}{dk}~R_{d,n}(k)= \left(\frac{\sqrt{n}-k}{\sqrt{n}}\right)^{\frac{1}{d}}\frac{S_{d,n}(k)}{2d^2n\sqrt{n-1}(k-\sqrt{n})^2},
\end{align*}
where $S_{d,n}(k)$ is given by
\begin{align*}
S_{d,n}(k)&:=-d^2k^2\sqrt{n-1}-2d^2kn +2d^2k\sqrt{n - 1}\sqrt{n}+2d^2\sqrt{n}\\&+2dkn-2dk\sqrt{n - 1}\sqrt{n}-2dk+k^2\sqrt{n - 1}.
\end{align*}

Notice the equivalence
\begin{align}\label{eqs}
S_{d,n}(k)>0\text{ if and only if }\frac{d}{dk}~R_{d,n}(k)>0.
\end{align}

We rewrite $S_{d,n}$ and provide an estimate for $0<k<1$: 
\begin{align*}
S_{d,n}(k)&=2d^2\sqrt{n}-2dk[(d-1)n-(d-1)\sqrt{n-1}\sqrt{n}+1]-k^2(d^2-1)\sqrt{n-1}\\
&>2d^2\sqrt{n}-2d[(d-1)n-(d-1)\sqrt{n-1}\sqrt{n}+1]-(d^2-1)\sqrt{n-1}\\&=S_{d,n}(1),
\end{align*}
which is possible, since $(d-1)n-(d-1)\sqrt{n-1}\sqrt{n}+1>0$.

Further analysis shows that 
\begin{align*}
S_{d,n}(1)&=2d^2\sqrt{n}-2d[(d-1)n-(d-1)\sqrt{n-1}\sqrt{n}+1]-(d^2-1)\sqrt{n-1}\\&\geq 0
\end{align*}
for $d\geq 2$ and $n\geq 1$, so that
\begin{align*}
S_{d,n}(k)>S_{d,n}(1)\geq 0
\end{align*}
holds for all $d\geq 2$, $n\geq 1$ and $0<k<1$.

By \eqref{eqs} this yields
\begin{align*}
\frac{d}{dk}~R_{d,n}(k)>0\text{ on }[0,1)
\end{align*}
for $d\geq 2$ and $n\geq 1$, which - together with \eqref{r0} - implies 
\begin{align*}
\frac{d}{dn}~Q_{d,k}(n)=R_{d,n}(k)>0
\end{align*}
for all $d\geq2$, $n\geq 1$ and $0<k<1$.
\end{proof}

The claim now yields inequality \eqref{lst}, which is equivalent to \eqref{k1}, and Proposition \ref{prop7} follows, since $0<\left(1-\frac{y_0}{x_0}\right)<1$.
\end{proof}

\subsection{Proof of Lemma 3}
With these preparations at hand, we are ready to prove our primary results. First, we use the auxiliary results from Subsection \ref{IIIa} to show Lemma \ref{lemma}:

\begin{proof}[Proof of Lemma 3]The claim of the theorem follows by induction over $n$:

\item[$\bullet$\quad The lower bound is trivial for $n\leq \frac{6d^2}{d^2-1};$] the upper bound holds for $n=1$, since\[
\frac{y_0}{x_0}\leq 1-\left(1-\frac{y_0}{x_0}\right)\frac{1}{\sqrt{1}}=\frac{y_0}{x_0}.\]

\item[$\bullet$]\quad For the inductive step, we consider the lower and the upper bounds separately.

\underline{Lower bound}:

Assume $\frac{y_{n-1}}{x_{n-1}}\geq 1-\left(\frac{\sqrt{6}d}{\sqrt{d^2-1}}\right)\frac{1}{\sqrt{n}}$ for some $n\geq \frac{6d^2}{d^2-1}$, then by Proposition \ref{prop6}:\[
\frac{y_{n-1}}{x_{n-1}}\geq 1-\left(\frac{\sqrt{6}d}{\sqrt{d^2-1}}\right)\frac{1}{\sqrt{n}}\geq K^d_{\sqrt{n+1}}(\frac{\sqrt{6}d}{\sqrt{d^2-1}}).\]

Given that, we can use Proposition \ref{prop5}:\[
\frac{y_{n-1}}{x_{n-1}}\geq K^d_{\sqrt{n+1}}(\frac{\sqrt{6}d}{\sqrt{d^2-1}})
\]\[\Longrightarrow\frac{y_{n}}{x_{n}}=g\left(\frac{y_{n-1}}{x_{n-1}}\right)\geq 1-\left(\frac{\sqrt{6}d}{\sqrt{d^2-1}}\right)\frac{1}{\sqrt{n+1}} \text{.}\]

\underline{Upper bound}:

Assume $\frac{y_{n-1}}{x_{n-1}}\leq 1-\left( 1-\frac{y_0}{x_0}\right)\frac{1}{\sqrt{n}}$ for some $n\geq 1$, then by Proposition \ref{prop7}:\[
\frac{y_{n-1}}{x_{n-1}}\leq 1-\left( 1-\frac{y_0}{x_0}\right)\frac{1}{\sqrt{n}}\leq K^d_{\sqrt{n+1}}\left(1-\frac{y_0}{x_0}\right).\]

Given that, we can use Proposition \ref{prop5}:\[
\frac{y_{n-1}}{x_{n-1}}\leq K^d_{\sqrt{n+1}}\left(1-\frac{y_0}{x_0}\right)
\]\[\Longrightarrow\frac{y_{n}}{x_{n}}=g\left(\frac{y_{n-1}}{x_{n-1}}\right)\leq 1-\left(1-\frac{y_0}{x_0}\right)\frac{1}{\sqrt{n+1}} \text{.}\]

This completes the induction and the proof of Lemma \ref{lemma}. 
\end{proof}

\section{Determining the critical 1-arm exponent}\label{IV}
Given the recursive respresentation from Theorem \ref{thm2}, our main result is an (almost) immediate consequence of Lemma \ref{lemma}, which was proven in the last section. All is left to do is to express Lemma \ref{lemma} in terms of the critical 1-arm exponent $\rho$ (as defined in \eqref{rho}). Therefore we proceed as follows:

\begin{proof}[Proof of Theorem \ref{thm1}]
Notice that for $0<y\leq x$
\begin{align*}
\frac{2}{x+1}-1\leq\frac{2}{y+1}-1.
\end{align*}

Thus, let $0<k_1<k_2$, such that   $\left(1-\frac{k_1}{\sqrt{n}}\right)\geq\frac{y_{n-1}}{x_{n-1}}\geq \left(1-\frac{k_2}{\sqrt{n}}\right)$ for $n>k_2^2$, then 
\begin{align}\label{cq}
\frac{2}{2-\frac{k_2}{\sqrt{n}}}-1\geq\frac{2}{\frac{y_{n-1}}{x_{n-1}}+1}-1\geq\frac{2}{2-\frac{k_1}{\sqrt{n}}}-1.
\end{align}

Such a choice of constants $k_1, k_2$ is possible for $d\geq2$ at $\beta=\beta_c$ by Lemma \ref{lemma}.

Furthermore, for  $0<x<1$ we can estimate:
\begin{align}\label{cqq}
1-x>\frac{2}{x+1}-1>\frac{1-x}{2}.
\end{align}

Combining \eqref{cq} and \eqref{cqq}, we obtain for $ d\geq 2,$ $\beta=\beta_c$ and $n>k^2_2:$
\begin{align*}
\frac{k_2}{\sqrt{n}}>\frac{2}{\frac{y_{n-1}}{x_{n-1}}+1}-1>\frac{k_1}{2\sqrt{n}}
\end{align*}
and the claim of the theorem follows by Theorem \ref{thm2}.
\end{proof}

\subsection*{Acknowledgement.} The authors are grateful to Aernout van Enter and Akira Sakai for commenting an earlier version of the manuscript. 

\bibliographystyle{plain}

\begin{thebibliography}
	a\bibitem[1]{a} M. Aizenman, Geometric analysis of $\phi^4$ fields and Ising models, \textit{Commun. Math. Phys. 86 (1982) 1-48}.  
  
	\bibitem[2]{1}M. Aizenman, H. Duminil-Copin and V. Sidoravicius, Random currents and continuity of Ising model's spontaneous magnetization, \textit{Commun. Math. Phys. 334 (2015) 719-742}.
	
	\bibitem[3]{b}M. Aizenman, R. Fernández, On the critical behavior of the magnetization in high-dimensional Ising models, \textit{J. Statist. Phys. 44 (1986) 393-454}.

	\bibitem[4]{e}R.J. Baxter, Exactly Solved Models in Statistical Mechanics, \textit{Academic Press (1982)}. 

	\bibitem[5]{2}R. Bissacot, M. Cassandro, L. Cioletti and E. Presutti, Phase Transitions in Ferromagnetic Ising Models with Spatially Dependent Magnetic Fields, \textit{Commun. Math. Phys. 336 No. 1 (2015) 41-53}.

	 \bibitem[6]{3}R. Bissacot and L. Cioletti, Phase Transition in Ferromagnetic Ising Models with Non-uniform External Magnetic Fields, \textit{J. Statist. Phys. 139 Issue 5 (2010) 769-778}.

	\bibitem[7]{f}R. Bissacot, E. Endo, A. van Enter, Stability of the phase transition of critical-field Ising model on Cayley trees under inhomogeneous external fields, \textit{Stoch. Process. Appl. 127 Issue 12 (2017) 4126-4138}.

	\bibitem[8]{d}S. Dommers, C. Giardinà and R. van der Hofstad, Ising critical exponents on Random trees and graphs, \textit{Commun. Math. Phys. 328 (2014) 355-395}.

	\bibitem[9]{c}J. Fröhlich, B. Simon, T. Spencer, Infrared bounds, phase transitions and continuous symmetry breaking, \textit{Commun. Math. Phys. 50 (1976) 79-95}. 

	\bibitem[10]{4}S. Handa, M. Heydenreich and A. Sakai, Mean-field bound on the 1-arm exponent for Ising ferromagnets in high dimensions, \textit{Preprint (2016) http://arxiv.org/abs/1612.08809}.

\bibitem[11]{5}E. Ising, Beitrag zur Theorie des Ferromagnetismus, \textit{Z. Phys. 31 (1925) 253-258}.

	\bibitem[12]{6} T.D. Lee and C.N. Yang, Statistical Theory of Equations of State and Phase Transitions II. Lattice Gas and Ising Model, \textit{Phys. Rev. 87 (1952) 404-409}.
	
  	 \bibitem[13]{7}R. Peierls, Ising´s model of ferromagnetism, \textit{Proc. Cambridge Phil. Soc. 32 (1936) 477-481}. 
  	 
   	\bibitem[14]{8} C.J. Preston, Gibbs States on Countable Sets, \textit{Cambridge University Press (1974)}.
   	
    \bibitem[15]{9} U.A. Rozikov, Gibbs Measures on Cayley Trees, \textit{World Scientific (2013)}.
    
	\bibitem[16]{10} C.N. Yang, The spontaneous magnetization of a two-dimensional Ising model, \textit{Phys. Rev. 85 (1952) 808-816}.


\end{thebibliography}
	
\end{document}